\documentclass{amsart}
\usepackage{amsmath,amssymb,amsthm,amsfonts,mathrsfs,mathtools}

\usepackage[utf8]{inputenc}
\usepackage{graphicx}
\usepackage[hyphens]{url}
\usepackage{caption}
\usepackage{hyperref}
\usepackage{tikz-cd}
\usepackage{url}

\theoremstyle{plain}
    \newtheorem{thm}{Theorem}[section]
    \newtheorem{lem}[thm]   {Lemma}
    \newtheorem{cor}[thm]   {Corollary}
\newcommand{\C}{\mathbb{C}}
\newcommand{\D}{\mathbb{D}}
\newcommand{\R}{\mathbb{R}}
\renewcommand{\H}{\mathcal{H}}
\newcommand{\V}{\mathcal{V}_n}
\newcommand{\Z}{\mathcal{Z}_n}
\newcommand{\T}{\mathcal{T}_n}
\newcommand{\UV}{\mathcal{UV}}

\newcommand{\Br}{\mathfrak{Br}}

\newcommand{\pa}[1]{\left(#1\right)}
\newcommand{\abs}[1]{\left|#1\right|}
\newcommand{\set}[1]{\left\{#1\right\}}
\newcommand{\surf}{\Sigma_{g,1}}
\newcommand{\Sone}{\mathbb{S}^1}

\renewcommand{\S}{\mathcal{S}}
\renewcommand{\phi}{\varphi}

\DeclareMathOperator{\Diff}{Diff}
\DeclareMathOperator{\Sq}{Sq}

\begin{document}

\title[Braid groups and mapping class groups]{Braid groups, mapping class groups and their homology with twisted coefficients}

\author{Andrea Bianchi}

\address{Mathematics Institute,
University of Bonn,
Endenicher Allee 60, Bonn,
Germany
}

\email{bianchi@math.uni-bonn.de}

\date{\today}


\begin{abstract}
We consider the Birman-Hilden inclusion
$\phi\colon\Br_{2g+1}\to\Gamma_{g,1}$
of the braid group into the mapping class group of an orientable surface with boundary, and prove that $\phi$ is stably trivial
in homology with twisted coefficients in the symplectic representation $H_1(\Sigma_{g,1})$ of the mapping class group;
this generalises a result of Song and Tillmann regarding homology with constant coefficients. Furthermore we
show that the stable homology of the braid group with coefficients in $\phi^*(H_1(\Sigma_{g,1}))$ has only $4$-torsion.
\end{abstract}

\maketitle

\section{Introduction}
Braid groups have a strong connection with mapping class groups of surfaces. On the one hand the braid group $\Br_n$
on $n$ strands is itself a mapping class group, namely the one associated to the surface $\Sigma_{0,1}^{n}$
of genus $0$ with one (parametrised) boundary component and $n$ (permutable) punctures.

On the other hand
Birman and Hilden show in \cite{BH2} that the group $\Br_{2g+1}$ can be identified with the hyperelliptic mapping class group:
this is a certain subgroup of the mapping class group $\Gamma_{g,1}$ of an orientable
surface of genus $g$ with one parametrised boundary component (see subsection \ref{subsec:hyp}).

It is natural to study the behaviour in homology of the Birman-Hilden inclusion $\phi\colon\Br_{2g+1}\to\Gamma_{g,1}$.
Song and Tillmann \cite{SongTillmann}, and later Segal and Tillmann \cite{SegalTillmann}, show that the map $\phi_*$ is stably trivial
in homology with constant coefficients. More precisely:
\begin{thm}
 \label{thm:ST}
 For any abelian group $A$ the map
 \[
  \phi_*\colon H_k(\Br_{2g+1};A)\to H_k(\Gamma_{g,1};A)
 \]
 is trivial for $k\leq \frac 23 g -\frac 23$.
\end{thm}
The range $k\leq \frac 23 g -\frac 23$ is the best known stable range for the homology with constant coefficients
of the mapping class group (see section \ref{sec:preliminaries} for the precise statement of Harer's stability theorem).

Both proofs of theorem \ref{thm:ST} use certain analogues of the maps $\phi$, namely the maps
\[
 \phi^{even}\colon \Br_{2g}\to\Gamma_{g-1,2}
\]
from a certain braid group on an even number of strands to a mapping class group of a surface with two boundary 
components. The maps $\phi^{even}$ can be put together to form a braided monoidal functor $\coprod_{g\geq 1}\Br_{2g}\to\coprod_{g\geq 1}\Gamma_{g-1,2}$.

Passing to classifying spaces of categories and then taking the group completion, one shows
that the stable map $\Br_{\infty}\to\Gamma_{\infty,2}$
behaves in homology as the restriction on $0$-th components
of a certain $\Omega^2$-map between the following $\Omega^2$-spaces.

The first $\Omega^2$-space is
$\Omega^2S^2$; its $0$-th connected component is the $\Omega^2$-space $\Omega^2 S^3$, which is the free $\Omega^2$-space over $S^1$.

The second $\Omega^2$-space is
\[
\Omega B\pa{\coprod_{g\geq 0}B\Gamma_{g,2}}\simeq \mathbb{Z}\times \pa{B\Gamma_{\infty}}^+,
\]
and in particular it has simply connected components, since $H_1(\Gamma_{\infty,2})=0$. Here $\pa{B\Gamma_{\infty}}^+$ denotes
the Quillen plus construction applied to the classifying space of the group $\Gamma_{\infty}$, which
is the colimit of the groups $\Gamma_{g,1}$ for increasing $g$ along the inclusions $\alpha$ (see subsection \ref{subsec:mcg}).

The map $\phi^{even}\colon\Omega^2 S^2\to \mathbb{Z}\times \pa{B\Gamma_{\infty}}^+$ is nullhomotopic on $\Omega^2S^3$,
because its restriction
to $S^1\subset\Omega^2 S^3$ is nullhomotopic. In particular the induced map in homology $\phi^{even}_*$
is trivial in degree $*>0$.

In \cite{BoT:Embeddings} B\"{o}digheimer and Tillmann generalise this argument to other
families of embeddings of braid groups into mapping class groups.

Our aim is to prove an analogue of theorem \ref{thm:ST} for homology with symplectic twisted coefficients.

\begin{thm}
\label{thm:STtwisted}
Consider the symplectic representation $\H\colon =H_1(\Sigma_{g,1})$, of the mapping class group $\Gamma_{g,1}$, and its
pull-back $\phi^*\H$, which is a representation of $\Br_{2g+1}$. The induced map in homology with twisted coefficients
\[
 \phi_*\colon H_k(\Br_{2g+1};\phi^*\H)\to H_k(\Gamma_{g,1};\H)
\]
is trivial for $k\leq \frac 23 g -\frac 23 -1$.
\end{thm}
Our proof is more elementary and only relies on a weak version of Harer's stability theorem: in
particular we will not need to stabilise with respect to the number of strands or the genus.

We also obtain a result concerning the homology $H_*(\Br_{2g+1};\phi^*\H)$ on its own:
\begin{thm}
 \label{thm:fourtorsion}
The homology $H_*(\Br_{2g+1};\phi^*\H)$ is $4$-torsion, i.e. every element vanishes when multiplied
by $4$.
\end{thm}
The homology $H_*(\Br_{2g+1};\phi^*\H)$ arises in a natural way as a direct summand of $H_*(\phi^*\S_{g,1})$. Here $\S_{g,1}$ denotes
the total space of the tautological $\Sigma_{g,1}$-bundle $\S_{g,1}\to B\Gamma_{g,1}$ over the classifying space of the
mapping class group, and $\phi^*\S_{g,1}$ is its pull-back on the braid group, or, as we have seen, on the hyperelliptic mapping class group.
This follows from the fact that every $\surf$-bundle has a section \emph{at the boundary} (see section \ref{sec:preliminaries}).

This article contains the main results of my Master Thesis \cite{Bianchi}. Recently Callegaro
and Salvetti (\cite{CallegaroSalvetti}) have computed explicitly the homology $H_*(\Br_{2g+1};\phi^*\H)$,
showing that it has even only $2$-torsion; in another work \cite{CallegaroSalvetti2} the same authors have
studied the analogue problem for totally ramified $d$-fold branched coverings of the disc.
Their results are partially based on results of my Master Thesis, which are discussed in this article.

I would like to thank Ulrike Tillmann, Mario Salvetti and Filippo
Callegaro for their supervision, their help and their encouragement during the preparation of my Master
Thesis, and Carl-Friedrich Bödigheimer for helpful discussions and detailed comments
on a first draft of this article.

\section{Preliminaries}
\label{sec:preliminaries}
In this section we recall some classical facts about braid groups and mapping class groups.
\subsection{Braid groups}
Let $\D=\set{z \,|\,\abs{z}<1}\subset\C$ be the open unit disc, and let
\[
 F_n(\D)=\set{(z_1,\dots,z_n)\in\D^n\;|\;z_i\neq z_j \;\forall i\neq j}
\]
be the \emph{ordered configuration space} of $n$ points in $\D$. There is a natural, free action of
$\mathfrak{S}_n$ on $F_n(\D)$, which permutes the labels of a configuration. The quotient space
is denoted by $C_n(\D)$ and is called the \emph{unordered configuration space} of $n$ points in $\D$.

Artin's braid group $\Br_n$ is defined as the fundamental group $\pi_1(C_n(\D))$; recall that
$C_n(\D)$ is an aspherical space (see \cite{FadellNeuwirth}), and hence a classifying space for $\Br_n$.

The braid group $\Br_n$ has a presentation (see \cite{Artin}) with generators $\sigma_1,\dots,\sigma_{n-1}$ and relations:
\begin{itemize}
 \item $\sigma_i\sigma_j=\sigma_j\sigma_i$ for $|i-j|\geq 2$;
 \item $\sigma_i\sigma_j\sigma_i=\sigma_j\sigma_i\sigma_j$ for $|i-j|=1$.
\end{itemize}

The space $C_n(\D)$ has a natural structure of complex manifold, with local coordinates $(z_1,\dots,z_n)$,
the positions of the points in the configuration. To stress that these local coordinates do not have a preferred
order, we will also write $\set{z_1,\dots,z_n}$.

\subsection{Mapping class groups}
\label{subsec:mcg}
Let $\Sigma_{g,m}$ be a smooth, oriented, compact surface of genus $g$ with $m\geq 1$
\emph{parametrised} boundary components.

We will be mainly interested in the case $m=1$, but we will need also the case $m=2$ to present
our results.

A parametrisation of the boundary is a diffeomorphism
$\partial\Sigma_{g,m}\cong \set{1,\dots,m}\times\Sone$, where $\Sone\in\C$ is the unit circle;
this diffeomorphism should induce on each boundary component the same orientation as the one
induced by the (oriented) surface $\Sigma_{g,m}$ on the boundary.

We choose as basepoint for $\Sigma_{g,m}$ the point $*\in\partial\Sigma_{g,m}$
corresponding to $(1,1)\in\set{1,\dots,m}\times\Sone$.

We consider the group $\Diff_{g,m}$ of diffeomorphisms $f\colon\Sigma_{g,m}\to\Sigma_{g,m}$ for which
there exists a collar neighborhood $U\subseteq \Sigma_{g,m}$ of the boundary $\partial\Sigma_{g,m}$ such that $f|_U$
is the identity.

This is a topological group with the Whitney $C^{\infty}$-topology. Note that
a diffeomorphism that fixes a neighborhood of the boundary (in particular an open set of $\Sigma_{g,m}$)
must be orientation-preserving. A result by Earle and Schatz (\cite{EarleSchatz}) ensures that $\Diff_{g,m}$
has contractible connected components, so the tautological map
\[
\Diff_{g,m}\to\pi_0(\Diff_{g,m})
\]
is a homotopy equivalence. The second term $\pi_0(\Diff_{g,m})$ is the discrete group of connected components
of $\Diff_{g,m}$: it is called the mapping class group of $\Sigma_{g,m}$ and it is denoted by $\Gamma_{g,m}$.

By taking classifying spaces we obtain a homotopy equivalence $B\Diff_{g,m}\simeq B\Gamma_{g,m}$.

The tautological action of $\Diff_{g,m}$ on $\Sigma_{g,m}$ yields, through the Borel construction, the map
\[
 E\Diff_{g,m}\times_{\Diff_{g,m}}\Sigma_{g,m}\to B\Diff_{g,m}=E\Diff_{g,m}/\Diff_{g,m}.
\]
This map is a fiber bundle with fiber $\Sigma_{g,m}$.
The pullback bundle along the inverse homotopy equivalence $B\Gamma_{g,m}\to B\Diff_{g,m}$ is denoted by
$p\colon \S_{g,m}\to B\Gamma_{g,m}$; we have a pull-back square
\[
 \begin{tikzcd}
  \S_{g,m} \ar[r,"\simeq"] \ar[d,"p"] & E\Diff_{g,m}\times_{\Diff_{g,m}}\Sigma_{g,m}\ar[d] \\
   B\Gamma_{g,m}\ar[r,"\simeq"] & B\Diff_{g,m}.
 \end{tikzcd}
\]

The fiber of $p$ is a surface diffeomorphic to $\Sigma_{g,m}$; boundaries of fibers are moreover equipped with
a parametrisation, so that the subspace $\partial\S_{g,m}$ given by the union of all boundaries of fibers is
canonically homeomorphic to $B\Gamma_{g,1}\times \pa{\set{1,\dots,m}\times\Sone}$; each boundary component
of each fiber of $p$ inherits the same orientation from the oriented fiber ($\simeq\Sigma_{g,m}$) to which it
belongs and from its identification with $\Sone$ along the aforementioned canonical homeomorphism.

The bundle $p$ is universal among bundles with all these
properties: if $X$ is a paracompact space and $\tilde p\colon\tilde S\to X$ is a $\Sigma_{g,m}$-bundle
over $X$ with a given homeomorphism between the subspace $\partial \tilde S$ of boundaries of fibers
with $X\times\pa{\set{1,\dots,m}\times\Sone}$, such that each boundary component of each fiber of $\tilde p$
inherits the same orientation
from the fiber to which it belongs and from the aforementioned homeomorphism,
then there is up to homotopy a unique classifying map $\psi\colon X\to B\Gamma_{g,m}$ such that $\tilde p\simeq \psi^*p$
as bundles with parametrised boundaries of fibers.

The bundle $p$ admits a global section \emph{at the boundary} $s_0\colon B\Gamma_{g,m}\to \S_{g,m}$,
obtained by choosing the basepoint of each fiber (i.e. the point corresponding to $(1,1)\in\set{1,\dots,m}\times\Sone$ under
the parametrisation). By abuse of notation, we will also see $B\Gamma_{g,m}=s_0(B\Gamma_{g,m})$ as a subspace
of $\S_{g,m}$.

Fibers of $p$ are smooth surfaces, and we can assemble together their tangent bundles to get a vector bundle
$\bar p^v\colon \mathcal{V}_{g,m}\to \S_{g,m}$ with fiber $\R^2$, called the \emph{vertical tangent bundle}.
Choosing a Riemannian metric on $\bar p^v$ and considering on each vector space its unit circle,
we can also define the \emph{unit vertical tangent bundle}
$p^v\colon \UV_{g,m}\to \S_{g,m}$, with fiber $\Sone$.

We can define a section of $p^v$ over the subspace
$\partial \S_{g,m}\simeq B\Gamma_{g,m}\times\pa{\set{1,\dots,m}\times\Sone}$: we assign to each point on the boundary
of some fiber of $p$ the unit vector which is tangent to that fiber, is orthogonal to the boundary of that fiber and points
outwards.
We will actually only need the restriction of this section to $B\Gamma_{g,m}=s_0(B\Gamma_{g,m})\subset\partial \S_{g,m}$:
we call it $s_0^v\colon B\Gamma_{g,m}=s_0(B\Gamma_{g,m})\to \UV_{g,m}$, and again by abuse of notation we see $B\Gamma_{g,m}$
as a subspace of $\UV_{g,m}$. See the following diagram

\[
\begin{tikzcd}[column sep=6em,row sep=3em]
  s_0^v(B\Gamma_{g,m}) \ar[d, equal]\ar[r,hook]
  & \UV_{g,m}\ar[d,"p^v"]\\
  s_0(B\Gamma_{g,m})\ar[r,hook]\ar[dr,equal]
  & \S_{g,m}\ar[d,"p"]\\
  & B\Gamma_{g,m}
 \end{tikzcd}
\]

The previous constructions are natural with respect to pullbacks: if $\tilde p\colon\tilde S\to X$ is a
$\Sigma_{g,m}$-bundle over a paracompact space $X$, we have a section \emph{at the boundary} $\tilde s_0\colon X\to \tilde S$,
a unit vertical tangent bundle $\tilde p^v\colon \tilde{\UV}\to \tilde S$ and a \emph{pointing outward}
section $\tilde s^v_0\colon X=\tilde s_0(X)\to \tilde{\UV}$.

We now restrict to the cases $m=1,2$ and construct a map $\beta\colon\Gamma_{g,1}\to\Gamma_{g,2}$. First
we decompose $\Sigma_{g,2}$ as the union of $\surf$ and a pair of pants $\Sigma_{0,3}$ along a boundary component.
Each diffeomorphism of $\surf$ fixing a collar neighborhood of $\partial\surf$ extends to a diffeomorphism
of $\Sigma_{g,2}$, by prescribing the identity map on $\Sigma_{0,3}$: we obtain a homomorphism
$\bar\beta\colon\Diff_{g,1}\to\Diff_{g,2}$, and the homomorphism $\beta$ is $\pi_0(\bar\beta)$. See figure \ref{fig:glue}.

\begin{figure}
 \centering
 \includegraphics[scale=0.7]{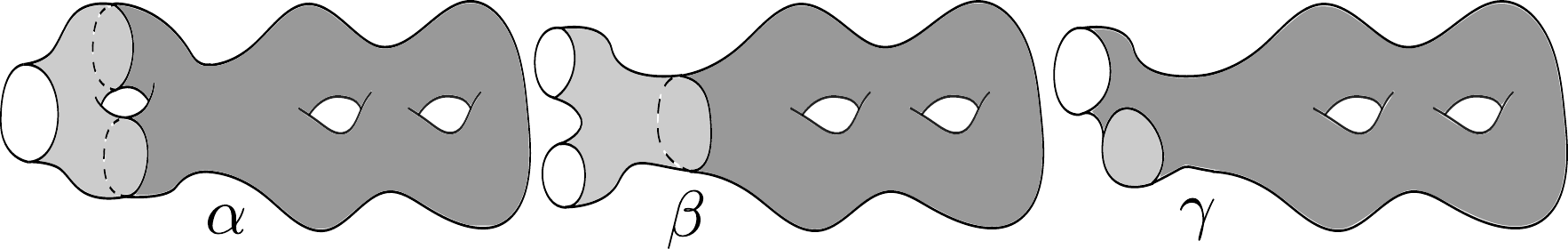}
 \caption{Glueing surfaces in different ways yields homomorphisms $\alpha$,$\beta$ and $\gamma$ between mapping class groups.}
 \label{fig:glue}
\end{figure}

Conversely, we can construct a map $\gamma\colon\Gamma_{g,2}\to \Gamma_{g,1}$ as follows. First, we decompose $\surf$ as
the union of $\Sigma_{g,2}$ and a disc $\Sigma_{0,1}$ along a boundary component. Each diffeomorphism of $\Sigma_{g,2}$
fixing a collar neighborhood of $\partial\Sigma_{g,2}$
extends to $\Sigma_{g,1}$, by prescribing the identity map on $\Sigma_{0,1}$: we obtain a homomorphism
$\bar\gamma\colon\Diff_{g,2}\to\Diff_{g,1}$, and the homomorphism $\gamma$ is $\pi_0(\bar\gamma)$.

The composition $\gamma\circ\beta\colon\Gamma_{g,1}\to\Gamma_{g,1}$ is essentially the identity: we are glueing a cylinder
$\Sigma_{0,2}\simeq\Sone\times[0,1]$ to $\partial\surf$ to obtain a surface that is again diffeomorphic to $\Sigma_{g,1}$;
moreover there is a preferred isotopy class of diffeomorphisms $\Sigma_{g,1}\to \Sigma_{g,1}\cup_{\Sone}\Sone\times[0,1]$,
represented by the evaluation at time 1 of any extension of the tautological isotopy
$\partial\surf\times [0,1]\overset{\simeq}{\to} \Sone\times[0,1]\subset\surf\cup_{\Sone}\Sone\times[0,1]$,
starting from the inclusion $\surf\subset\surf\cup_{\Sone}\Sone\times[0,1]$. We can thus identify the mapping class
group of $\Sigma_{g,1}$ and the mapping class group of $\Sigma_{g,1}\cup_{\Sone} \Sone\times[0,1]$, and under
this identification the map $\gamma\circ\beta$ is the identity of $\Gamma_{g,1}$.

Finally, consider the following morphism
of groups $\alpha\colon\Gamma_{g,2}\to\Gamma_{g+1,1}$: this time we obtain $\Sigma_{g+1,1}$ glueing
$\surf$ and pair of pants $\Sigma_{0,3}$ along two boundary components. Again we
get first a homomorphism $\Diff_{g,2}\to\Diff_{g+1,1}$ and then a homomorphism $\alpha$ between the corresponding
mapping class groups.

We will state Harer's stability theorem in a form that suffices for our purposes
(see \cite{Harer} for the original theorem and 
\cite{Boldsen, ORW:resolutions_homstab} for the improved stability ranges).
\begin{thm}[Harer]
\label{thm:Harer}
Let $A$ be an abelian group. The maps $\alpha,\beta,\gamma$ described above induce isomorphisms in homology in a certain range:
\[
 \alpha_*\colon H_k(\Gamma_{g,2};A)\cong H_k(\Gamma_{g+1,1};A)\quad \mbox{for }k\leq \frac 23 g-\frac 23;
\]
\[
 \beta_*\colon H_k(\Gamma_{g,1};A)\cong H_k(\Gamma_{g,2};A)\quad \mbox{for }k\leq \frac 23 g;
\]
\[
 \gamma_*\colon H_k(\Gamma_{g,2};A)\cong H_k(\Gamma_{g,1};A)\quad \mbox{for }k\leq \frac 23 g.
\]
\end{thm}
Theorem \ref{thm:ST} relies on the full statement of theorem \ref{thm:Harer}, but
in the proof of theorem \ref{thm:STtwisted}
we will only need
homological stability for the maps $\beta$ and $\gamma$:
these are the stabilisation maps that change the number of boundary components but not the genus.

We will also need the following classical result (see \cite{FarbMargalit}, propositions 3.19 and 4.6)
\begin{thm}
 The space $\UV_{g,1}$ is a classifying space for $\Gamma_{g,2}$, i.e. it is homotopy equivalent to
 $B\Gamma_{g,2}$.
 
 The map $s_0^v\circ s_0\colon B\Gamma_{g,1}\to\UV_{g,1}$ induces the map $\beta$
 on fundamental groups.
 
 The map $p\circ p^v\colon \UV_{g,1}\to B\Gamma_{g,1}$ induces the map $\gamma$ on fundamental groups.
\end{thm}

\subsection{Hyperelliptic mapping class groups}
\label{subsec:hyp}
Fix a diffeomorphism $J$ of $\Sigma_{g,1}$ with the following properites:
\begin{itemize}
 \item $J^2$ is the identity of $\Sigma_{g,1}$;
 \item $J$ acts on $\partial\Sigma_{g,1}\cong\Sone$ as the rotation by an angle $\pi$;
 \item $J$ has exactly $2g+1$ fixed points in the interior of $\Sigma$.
\end{itemize}
The quotient $\Sigma_{g,1}/J$ is a disc and the
map $\Sigma_{g,1}\to\Sigma_{g,1}/J$ is a $2$-fold branched covering map with $2g+1$ branching points.
We say that $J$ is a \emph{hyperelliptic involution} of $\Sigma_{g,1}$.

Consider the group $\Diff_{g,1}^{ext}$ of diffeomorphisms $f\colon\Sigma_{g,1}\to\Sigma_{g,1}$
that preserve the orientation and restrict on a neighborhood of $\partial\Sigma_{g,1}$ either
to the identity, or to $J$. We have a short exact sequence of topological groups
\[
 \begin{tikzcd}
  1\ar[r] & \Diff_{g,1} \ar[r] & \Diff_{g,1}^{ext}\ar[r] & \mathbb{Z}_2 \ar[r] &1.
 \end{tikzcd}
\]
There is a section $\mathbb{Z}_2\to\Diff_{g,1}^{ext}$ given by $J$. Taking connected
components we obtain a split short exact sequence
\[
 \begin{tikzcd}
  1\ar[r] & \Gamma_{g,1} \ar[r] & \Gamma_{g,1}^{ext}\ar[r] & \mathbb{Z}_2 \ar[r] &1,
 \end{tikzcd}
\]
where $\Gamma_{g,1}^{ext}=\pi_0\pa{\Diff_{g,1}^{ext}}$ is called the \emph{extended mapping class group}.

The \emph{extended hyperelliptic mapping class group} $\triangle_{g,1}^{ext}$ is the centralizer in $\Gamma_{g,1}^{ext}$
of the mapping class of $J$:
\[
 \triangle_{g,1}^{ext}=Z([J]).
\]
The \emph{hyperelliptic mapping class group} $\triangle_{g,1}$ is the intersection in $\Gamma_{g,1}^{ext}$
between $\Gamma_{g,1}$ and $\triangle_{g,1}^{ext}$. We have an isomorphism
\[
 \triangle^{ext}_{g,1}\simeq \triangle_{g,1}\times\left<J\right>.
\]

\section{Definition of the map $\phi$ and a general construction}
\label{sec:defphi}
We consider on $\Sigma_{g,1}$ a chain of $2g$ simple closed curves $c_1,\dots,c_{2g}$, such that
$c_i\cap c_j=\emptyset$ for $\abs{i-j}\geq 2$, whereas $c_i$ and $c_j$ intersect transversely in
one point if $\abs{i-j}=1$. Note that
a tubular neighborhood of the union of these curves is itself diffeomorphic to $\surf$.
See figure \ref{fig:birman}.

\begin{figure}
 \centering
 \includegraphics[scale=0.6]{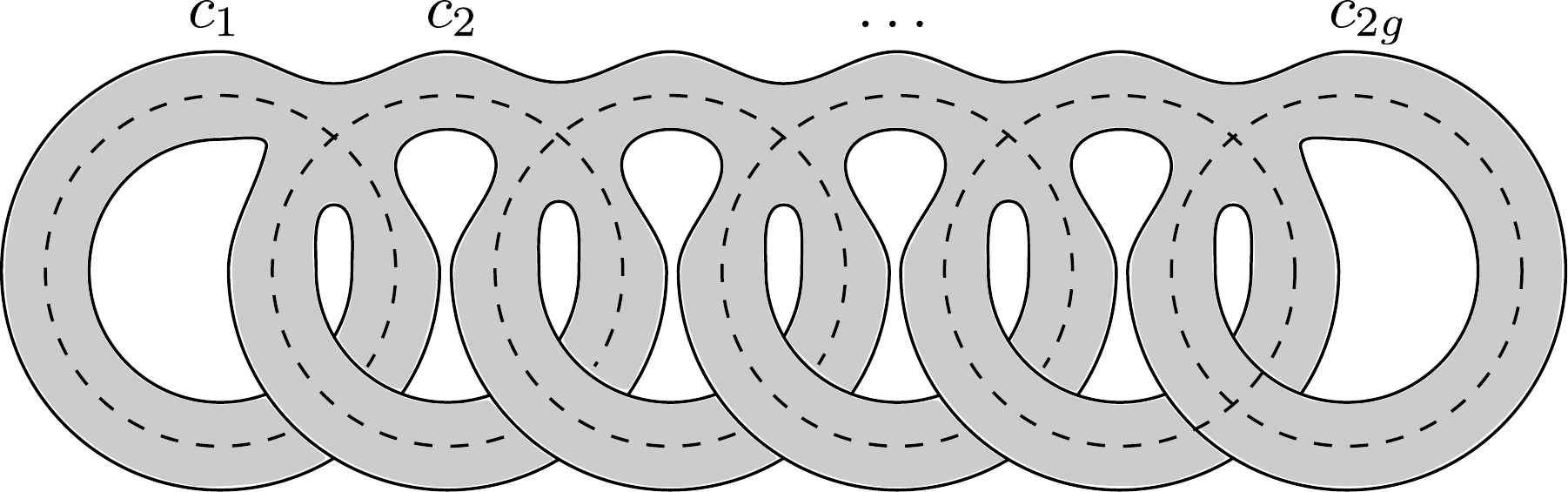}
 \caption{A chain of $2g$ simple closed curves on $\Sigma_{g,1}$.}
 \label{fig:birman}
\end{figure}

Denote by $D_i\in\Gamma_{g,1}$ the Dehn twist about the curve $c_i$; then it is a classical
result (see \cite{FarbMargalit}, Fact 3.9 and Proposition 3.11) that $D_iD_j=D_jD_i$
in $\Gamma_{g,1}$ for $\abs{i-j}\geq 2$, and $D_iD_jD_i=D_jD_iD_j$ for $\abs{i-j}=1$.

Therefore there is an induced morphism of groups
\[
 \phi\colon\Br_{2g+1}\to\Gamma_{g,1}
\]
which is defined by mapping the generator $\sigma_i\in\Br_{2g+1}$ to the Dehn twist $D_i\in\Gamma_{g,1}$. This map is
called the Birman-Hilden inclusion: it is indeed injective and its image is the hyperelliptic
mapping class group (see \cite{BH1,BH2}).

From now on let $n\colon=2g+1$, in particular $n$ is odd.
We give now a nice, geometric description of the $\surf$-bundle $\phi^*\S_{g,1}$ over $C_n(\D)\simeq B\Br_n$.

Consider, in the complex manifold $C_n(\D)\times\overline{\D}\times\C$,
the subspace
\[
 \V=\set{\pa{\set{z_1,\dots,z_n},x,y} \;|\; y^2=\prod_{i=1}^n(x-z_i)}.
\]
Here $\overline\D$ is the closed unit disc in $\C$.
First we show that $\V$ is a smooth manifold with boundary: indeed it is the zero locus on $C_n(\D)\times\overline\D\times\C$ of the
function $f(\set{z_i},x,y)=y^2-\prod_i(x-z_i)$, whose partial derivatives with respect to $x$ and $y$ are
\[
\frac {df}{dx}(\set{z_1,\dots,z_n},x,y)=-\sum_{i=1}^n\prod_{j\neq i}(x-z_j) 
\]
\[
 \frac {df}{dy}(\set{z_1,\dots,z_n},x,y)=2y
\]
If $\frac {df}{dy}$ vanishes, then $y=0$; if moreover $f$ vanishes, then $x=z_i$ for exactly one value
of $i$. Then all the summands but exactly one in the sum for $\frac {df}{dx}$ vanish, and therefore $\frac {df}{dx}\neq 0$.

We have thus shown that $\V$ is a smooth manifold, as $df$ never vanishes
on $\V$. Moreover at least one of the $x$ and the $y$ component of $df$ does not vanish on $\V$, therefore
the natural projection
\[
\pi\colon\V\to C_n(\D)
\]
is a submersion; in particular its fibers are smooth.
Note also that $\V$ is transverse to $C_n(\D)\times\partial\overline\D\times\C$:
if $|x|=1$ then $x\neq z_i$ for all $i$ and we can rewrite
\[
\frac {df}{dx}(\set{z_1,\dots,z_n},x,y)=-\pa{\prod_{i=1}^n(x-z_i)}\pa{\sum_{i=1}^n\frac{1}{x-z_i}}\neq 0,
\]
where the sum is non-zero because it has a non-trivial component in the direction of $\frac 1x$: if we consider
summands as vectors in $\R^2$ with the usual scalar product, then each summand has a positive scalar product
with the vector $\frac 1x$.

The fibers of $\pi$ are smooth manifolds of complex dimension 1, i.e. Riemann surfaces.
By projecting to $x$, each fiber is a double covering of $\overline{\D}$, branched over $n$ points: the covering map is
given by the projection on the $x$ coordinate. Thus the Euler characteristic of the fiber is
$2\cdot\chi(\overline{\D})-n=1-2g$. The boundary of the fiber over any $q=\set{z_1,\dots,z_n}\in C_n(\D)$
is $\set{(x,y)\;|\;\abs x=1, y^2=\prod_{i=1}^n(x-z_i)}$ and is a \emph{connected} double covering of $\Sone$:
a section of this covering would be a continuous choice, for $x\in\Sone$, of a square root $y=\sqrt{\prod_{i=1}^n(x-z_i)}$,
which does not exist since $n$ is odd.

Therefore the fiber of $\pi$ is diffeomorphic to $\surf$.

We want to parametrise the boundary component of each fiber. For any $q=\set{z_1,\dots,z_n}$
we can consider the equation $y^2=\prod_{i=1}^g(1-tz_i)$, for $t$ ranging in $[0,1]$. If $t=1$
the two solutions for $y$ give rise to two points $p_1,p_2\in\partial \pi^{-1}(q)$, putting $x=1$;
if $t=0$ the two values of $y$ are $\pm 1$. As $\prod_{i=1}^g(1-tz_i)\neq 0$ for all $t$, the two
values of $y$ are always different and change continuously while $t$ ranges from $0$ to $1$. This
gives a bijection of the sets $\set{p_1,p_2}$ and $\set{\pm 1}$. Assume that $p_1$ corresponds to $+1$; then
we parametrise $\partial\pi^{-1}(q)$ with the unique continuous choice of a square root $\sqrt{x}$ taking the
value $+1$ on $p_1$. This construction is continuous in $q\in C_n(\D)$.

We have therefore constructed a $\surf$-bundle over $C_n(\D)$, and this yields a classifying map
$C_n(\D)\to B\Gamma_{g,1}$ which in turn gives a map $\Br_n\to\Gamma_{g,1}$ between fundamental groups:
the induced map is precisely $\phi$ (see also \cite{SegalTillmann}).

The last construction admits a slight generalisation, that we briefly discuss here.

Let $B$ be a topological space and let $\psi\colon B\to \C [x,y]$ be a continuous function
from $B$ to the space of polynomials in two variables. Assume the following:
\begin{itemize}
 \item there exist two relatively prime positive integers $m,n$ such that
 for every $b\in B$ the polynomial $\psi(b)(x,y)$ has the form $\pm x^n\pm y^m +$lower order terms, all of
 order $<n$ in $x$ and $<m$ in $y$;
 \item for every $b\in B$ there is no point $(x,y)\in\C^2$ where all the following functions vanish: $\psi(b)(x,y)$,
 $\frac d{dx}\psi(b)(x,y)$
 and $\frac d{dy}\psi(b)(x,y)$;
 \item for every $b\in B$ there is no point $(x,y)\in\C^2$ with $\abs{x}\geq 1$ where both $\psi(b)(x,y)$ and
 $\frac d{dx}\psi(b)(x,y)$ vanish.
\end{itemize}
Then we can consider in $B\times \overline\D\times\C$ the zero locus of $\psi$, that is, the set
\[
 \mathcal{V}=\set{(b,x,y)\,|\,\psi(b)(x,y)=0}.
\]
There is a natural projection of $\mathcal{V}$ onto $B$, and each fiber is a smooth surface of genus $\frac{(m-1)(n-1)}2$
with one boundary component. This boundary component is a $m$-fold covering of $\Sone$ by projecting on $x$,
and it can be parametrised with the parameter $\sqrt[d]{x}$ starting from a point $(1,y_0)$ obtained again
by shrinking to zero all lower order terms of the polynomial (and considering $1$ as the preferred $m$-th root
of the unity).

The construction of $\V$ is a special case of this construction, in which $B=C_n(\D)$ and $\psi(\set{z_1,\dots,z_n})=y^2-\prod_{i=1}^n(x-z_i)$.

Another example is given, again for $B=C_n(\D)$ and for a fixed $d\geq 3$, by the assigment $\psi(\set{z_1,\dots,z_n})=y^d-\prod_{i=1}^n(x-z_i)$:
one gets the universal family of \emph{superelliptic curves} of degree $d$.

A superelliptic curve of degree $d$ is a $d$-fold covering of $\overline{\D}$
branched over $n$ points; its group of deck transformations is cyclic of order $d$ (in particular it acts transitively on fibers); the fiber over
each branching point consists of only one point, and all branching points have the same total holonomy, computed with respect to any regular point.
This family was studied by Callegaro and Salvetti in \cite{CallegaroSalvetti2}.

\section{Unit vertical vector fields}
Our next aim is to construct on the $\surf$-bundle $\V\to C_n(\D)$ a unit vertical vector field, i.e.
a section of the $\Sone$-bundle $\phi^*\UV_{g,1}\to\V=\phi^*\S_{g,1}$. To do so consider on the entire manifold
$C_n(\D)\times\overline{\D}\times\C$ the holomorphic vector field
\[
 \vec v(\set{z_1,\dots,z_n},x,y)=\frac {df}{dy}\cdot\frac{\partial}{\partial x} - \frac {df}{dx}\cdot\frac{\partial}{\partial y}=
 2y\cdot \frac{\partial}{\partial x}-\pa{\sum_{i=1}^n\prod_{j\neq i}(x-z_j)}\cdot\frac{\partial}{\partial y}.
\]
Then $\vec v$ does not vanish on $\V$: we have already seen that on each point of $\V$ at least one of the $x$- and $y$-partial derivatives
of $y^2-\prod_{i=1}^n(x-z_i)$ does not vanish. Moreover $\vec v$ is tangent to $\V$, since it
annihilates $df$; and $\vec v$ is vertical, as it is a linear
combination of $\frac{\partial}{\partial x}$ and $\frac{\partial}{\partial y}$.

Therefore, up to the canonical identification
between the holomorphic tangent bundle and the real tangent bundle and up to renormalisation, we have found
a unit vertical vector field on $\V$, i.e. a section of the $\Sone$-bundle $\phi^*\UV_{g,1}\to\phi^*\S_{g,1}$.
We already have a unit vertical vector field $\phi^*s_0^v$ on the subspace $C_n(\D)=\phi^*s_0(C_n(\D))\subset \phi^*\S_{g,1}$,
and the ratio between them (in the sense of ratio between sections of a principar $\Sone$-bundle)
is given by a map $\theta\colon C_n(\D)\to \Sone$; if we substitute $\vec v$ with
$\theta\cdot \vec v$, then our global vertical vector field extends the canonical one over $\phi^*s_0(C_n(\D))$.

The same construction works in the generalised framework introduced at the end of section \ref{sec:defphi}:
this time $\vec v$ is given on each fiber by the formula
\[
 \vec v(b,x,y)=\frac {d\psi(b)}{dy}(x,y)\cdot \frac{\partial}{\partial x} - \frac {d\psi(b)}{dx}(x,y)\cdot \frac{\partial}{\partial y}
\]
and again we can modify it so as to agree with the canonical vector field over the section at the boundary.

We sketch now an alternative proof of the existence of a unit vertical vector field on $\V\to C_n(\D)$
extending the canonical one over the section at the boundary. Let $\vec{\mathfrak{v}}$ be a vector field on $\Sigma_{g,1}$ as in
figure \ref{fig:field}: it is orthogonal to the curve $c_1$, parallel to $c_2$, again orthogonal to $c_3$ and so on;
moreover if $*\in\Sigma_{g,1}$ denotes the basepoint, then $\vec{\mathfrak{v}}(*)$ is exactly the unit tangent vector
at $*$ that is orthogonal to $\partial\Sigma_{g,1}$ and points outwards.

Let $\mathbb{V}$ be the space of all vector fields $\vec w$ on $\Sigma_{g,1}$ that satisfy $\vec w(*)=\vec{\mathfrak{v}}(*)$
and that have no zeroes on $\Sigma_{g,1}$
(we say briefly that they are \emph{non-vanishing}). Then
$\mathbb{V}\simeq Map_*(\Sigma_{g,1};\Sone)$ is a disjoint union of infinitely many contractible components.
The group $\Diff_{g,1}$ acts on $\mathbb{V}$ through
differentials of diffeomorphisms: this action is well-defined thanks to the hypothesis that differomorphisms in $\Diff_{g,1}$ restrict to the identity
on a neighborhood of the boundary, so that in particular their differential fixes the vector $\vec{\mathfrak{v}}(*)$.

There is an induced action of the mapping class group $\Gamma_{g,1}$ on $\pi_0(\mathbb{V})$, and the \emph{framed
mapping class group} associated to $\vec{\mathfrak{v}}$ is by definition the stabiliser of $[\vec{\mathfrak{v}}]\in \pi_0(\mathbb{V})$, which is a subgroup
$\Gamma_{g,1}^{fr}(\vec{\mathfrak{v}})\hookrightarrow\Gamma_{g,1}$; we call $\mathfrak{i}$ this inclusion of groups.

Consider $\mathfrak{i}^*\S_{g,1}\to B\Gamma_{g,1}^{fr}(\vec{\mathfrak{v}})$, the pull-back of the universal
surface bundle along the map $\mathfrak{i}\colon B\Gamma_{g,1}^{fr}(\vec{\mathfrak{v}})\to B\Gamma_{g,1}$.
Using that connected components of $\mathbb{V}$ are contractible one can construct a unit vertical vector field $\mathfrak{v}$
on $\mathfrak{i}^*\S_{g,1}$ that restricts to $\mathfrak{i}^*s_0^v$ on the section at the boundary $\mathfrak{i}^*s_0(B\Gamma_{g,1}^{fr}(\vec{\mathfrak{v}}))$.

The key remark is now that the image of the Birman-Hilden inclusion $\phi\colon\Br_{2g+1}\to\Gamma_{g,1}$ lies inside $\Gamma_{g,1}^{fr}(v_0)$: indeed
the vector field $\vec{\mathfrak{v}}$ is preserved by the differential of all Dehn twists about the curves $c_i$, up to isotopy through vector fields in $\mathbb{V}$.

Therefore the map $\phi\colon C_n(\D)\to B\Gamma_{g,1}$ factors through $B\Gamma_{g,1}^{fr}(\vec{\mathfrak{v}})$, and we can now pullback
the unit vertical vector field $\mathfrak{v}$ over $\mathfrak{i}^*\S_{g,1}$ to a unit vertical vector field $\vec v$
over $\phi^*\S_{g,1}$with all the desired properties.

\begin{figure}
 \centering
 \includegraphics[scale=0.65]{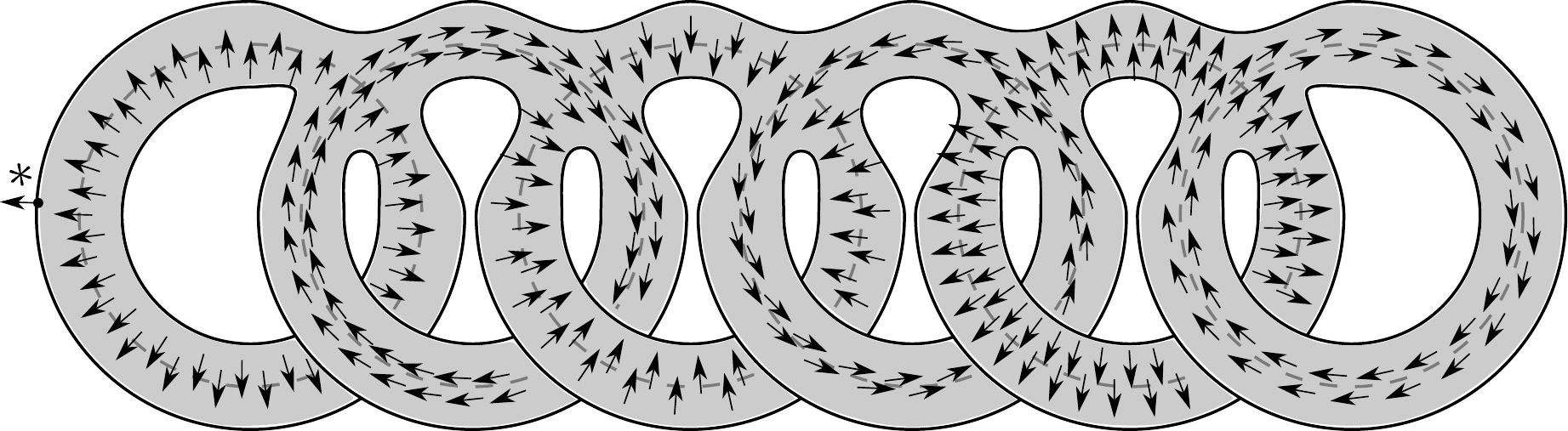}
 \caption{The vector field $\vec{\mathfrak{v}}$.}
 \label{fig:field}
\end{figure}

\section{Stable vanishing of $\phi_*$}
Our proof of theorem \ref{thm:STtwisted} consists of two steps. In the first step we formulate
the problem in an alternative way, namely we replace the map
\[
 \phi_*\colon H_k(\Br_{2g+1};\phi^*\H)\to H_k(\Gamma_{g,1};\H)
\]
with the map
\[
 \phi_*\colon H_{k+1}(\phi^*\S_{g,1},B\Br_{2g+1})\to H_{k+1}(\S_{g,1},B\Gamma_{g,1}).
\]
Recall that $B\Br_{2g+1}$ can be seen as a subspace of $\phi^*\S_{g,1}$ through the section at the boundary.

The second map deals only with homology with constant coefficients, although we have now more complicated
spaces.

In the second step we factor the above map through the homology groups
\[
 H_{k+1}(\UV_{g,1},B\Gamma_{g,1})
\]
which are trivial in the stable range. Thus the map $\phi_*$ is trivial in homology in the stable range.
The strategy of the proof is summarized in the following diagram
\[
 \begin{tikzcd}[column sep=2em,row sep=3em]
 H_k(\Br_{2g+1};\phi^*\H)\ar[rr,"\phi_*"]\ar[d,dashed,"\cong"] & & H_k(\Gamma_{g,1};\H)\ar[d,dashed,"\cong"]\\
 H_{k+1}(\phi^*\S_{g,1},B\Br_{2g+1})\ar[rr,"\phi_*"]\ar[dr,dashed] & & H_{k+1}(\S_{g,1},B\Gamma_{g,1})\\
  &  H_{k+1}(\UV_{g,1},B\Gamma_{g,1})=0.\ar[ur,"p^v_*"] &
 \end{tikzcd}
\]

\subsection{The reformulation of the problem.}
The bundle $\S_{g,1}\to B\Gamma_{g,1}$, together with the global section $s_0$, can be seen as a pair of bundles
$(\S_{g,1},B\Gamma_{g,1})\to B\Gamma_{g.1}$ with fiber the pair $(\surf,*)$. There is an associated
Serre spectral sequence whose second page contains the homology groups
\[
 E^2_{p,q}=H_p(B\Gamma_{g,1};H_q(\surf,*))
\]
and whose limit is the homology of the pair $(\S_{g,1},B\Gamma_{g,1})$. Note that the homology
group $H_q(\surf,*)$ is non-trivial only for $q=1$, in which case it is exactly the symplectic representation
$\H$ of $\Gamma_{g,1}$. So the second page of the spectral sequence has only one non-vanishing row and therefore coincides with
its limit, i.e.
\[
 H_{p+1}(\S_{g,1},B\Gamma_{g,1})=H_p(B\Gamma_{g,1};\H).
\]
The whole construction is natural with respect to pullbacks. Let again $n=2g+1$. The natural map
$\phi\colon(\phi^*\S_{g,1},B\Br_n)\to (\S_{g,1},B\Gamma_{g,1})$ is a map of pairs of bundles,
i.e. it covers the map $\phi\colon B\Br_n\to B\Gamma_{g,1}$. The fiber of the pair of bundles
$(\phi^*\S_{g,1},B\Br_n)\to B\Br_n$ is still the pair $(\surf, *)$, so its homology is concentrated
in degree one and the corresponding spectral sequence gives again an isomorphism
\[
  H_{p+1}(\phi^*\S_{g,1},B\Br_n)=H_p(B\Br_n;\phi^*\H).
\]
The induced map between the second pages of the spectral sequences is the map
\[
\phi_*\colon H_k(\Br_n;\phi^*\H)\to H_k(\Gamma_{g,1};\H),
\]
appearing in theorem \ref{thm:STtwisted}; the induced map on the limit is the map
\[
 \phi_*\colon H_{k+1}(\phi^*\S_{g,1},B\Br_n)\to H_{k+1}(\S_{g,1},B\Gamma_{g,1}).
\]
Hence we can study the latter map, thus reducing ourselves to understand the behaviour
of the map of pairs $\phi\colon (\phi^*\S_{g,1},B\Br_n)\to (\S_{g,1},B\Gamma_{g,1})$
in homology with constant coefficients.

\subsection{The factorisation through $H_{k+1}(\UV_{g,1},B\Gamma_{g,1})$.}
Recall that there is a unit vertical vector field $\vec v$ on $\V=\phi^*\S_{g,1}$
extending the canonical vector field $\phi^*s_0^vS$ on the subspace $B\Br_n\subset \phi^*\S_{g,1}$.

This means that in the following diagram
\[
\begin{tikzcd}[column sep=6em,row sep=3em]
  \pa{\phi^*\UV,B\Br_n} \ar[r,"\phi"]\ar[d,"\phi^*p^v"] & \pa{\UV, B\Gamma_{g,1}} \ar[d,"p_v"]\\
  \pa{\phi^*\S_{g,1},B\Br_n} \ar[r,"\phi"] \ar[ur,dashed]  & \pa{\S_{g,1},B\Gamma_{g,1}}
 \end{tikzcd}
\]
there is a dashed diagonal arrow lifting the bottom horizontal map, so that the lower right triangle commutes. In particular the map 
\[
 \phi_*\colon H_{k+1}(\phi^*\S_{g,1},B\Br_n)\to H_{k+1}(\S_{g,1},B\Gamma_{g,1})
 \]
factors through the homology group $H_{k+1}\pa{\UV, B\Gamma_{g,1}}$. Since by theorem \ref{thm:Harer}
the inclusion $s_0^v\circ s_0\colon B\Gamma_{g,1}\to \UV_{g,1}$ is a homology-isomorphism in degree $\leq \frac 23 g$,
we deduce that $H_{k+1}\pa{\UV, B\Gamma_{g,1}}=0$ for $k+1\leq \frac 23 g$, and therefore for $k\leq \frac 23 g -1$
the map
\[
 \phi_*\colon H_{k+1}(\phi^*\S_{g,1},B\Br_n)\to H_{k+1}(\S_{g,1},B\Gamma_{g,1})
\]
is the zero map. This completes the proof of theorem \ref{thm:STtwisted}.

The result can be generalised
to the case in which we construct a $\surf$-bundle over a space $B$ through a map $\psi\colon B\to\C[x,y]$ as in
the previous section. We obtain a map $\Psi\colon B\to B\Gamma_{g,1}$ that induces the trivial map
\[
 \Psi_*\colon H_k(B;\Psi^*\H)\to H_k(B\Gamma_{g,1};\H)
\]
in homology in degree $k\leq \frac 23 g-1$. The proof is the same.

\section{Torsion property of $H_*(\Br_n;\phi^*\H)$.}
In this section we prove theorem \ref{thm:fourtorsion}. Using the isomorphism
\[
 H_k(\Br_n;\phi^*\H)\simeq H_{k+1}(\V,C_n(\D))
\]
we want to prove that the second group is $4$-torsion.

On the complex manifold $\V$ we consider the holomorphic function $y$. We call $\Z\subset\V$ the
zero locus of $y$: $\Z$ is a smooth complex
submanifold, indeed the vector field already considered
\[
\vec v=2y\cdot \frac{\partial}{\partial x}-\pa{\sum_{i=1}^n\prod_{j\neq i}(x-z_j)}\cdot\frac{\partial}{\partial y}
\]
is non-zero on the whole $\V$, and thererfore on $\Z$ its $y$-component must be non-zero; this witnesses
the non-vanishing of $dy|_{\V}$ on $\Z$. As $\Z$ is the smooth zero-locus of
a holomorphic function on the complex manifold $\V$, the normal bundle of $\Z$ in $\V$ must be trivial.

The space $\Z$ is homeomorphic to the space
\[
 C_{n-1,1}(\D)=\set{(\set{z_1,\dots, z_{n-1}},x)\in C_{n-1}(\D)\times \D\;|\; x\neq z_i\;\forall 1\leq i\leq n-1},
\]
that we call the \emph{configuration space} of $n-1$ \emph{black} and one \emph{white} points in the disc.

Indeed if $y=0$, then the equation $y^2=\prod_{i=1}^n(x-z_i)$ defining $\V$ tells us that $x$ must coincide
with one, and exactly one, of the numbers $z_i$; hence a point of $\Z$ is exactly an unordered configuration of $n$ points in $\D$,
one of which is special (and we say, it is \emph{white}) because it coincides with $x$.

We call $\T$ the (open) complement of $\Z$ in $\V$.

We call $N(\Z)$ a small, closed tubular neighborhood of
$\Z$ in $\V$. Since the normal bundle of $\Z$ in $\V$ is trivial, we have $N(\Z)\cong \Z\times\overline{\D}$,
and $N(\Z)\cap\T\simeq \partial N(\Z)\cong\Z\times\Sone$. By construction the copy of $C_n$ contained in $\V$, i.e. the
image of the section $\phi^*s_0$, is contained in $\T\setminus N(\Z)$.

We have a Mayer-Vietoris sequence
\[
 \dots \rightarrow H_k\pa{ N(\Z)\cap\T} \rightarrow H_k\pa{\Z}\oplus H_k\pa{\T,C_n(\D)} \rightarrow H_k\pa{\V,C_n(\D)} \rightarrow \dots
\]
from which we derive the following lemma.
\begin{lem}
 \label{lemma:MVsuTZ}
There is a long exact sequence
 \[
 \begin{array}{c}
 \dots \rightarrow H_k\pa{C_{n-1,1}(\D)}\oplus H_{k-1}\pa{C_{n-1,1}(\D)}\otimes H_1(\Sone) \overset{\iota}{\rightarrow}\\
  \overset{\iota}{\rightarrow} H_k\pa{C_{n-1,1}(\D)}\oplus H_k\pa{\T,C_n(\D)} \rightarrow H_k\pa{\V,C_n(\D)} \rightarrow \dots
  \end{array}
\] 
\end{lem}

Our goal is to get information about the homology of $\V$ by knowing the other homologies and
the behaviour of the maps in the previous sequence.
In particular we need some results about the space $\T$.

 There is a double convering map $\Sq\colon\T\to C_{n,1}(\overline{\D})$,
where
\[
 C_{n,1}(\overline{\D})=\set{(\set{z_1,\dots, z_n},x)\in C_n(\D)\times \overline{\D}\;|\; x\neq z_i\;\forall 1\leq i\leq n}.
\]
The map $\Sq$ is given by forgetting the value of $y$ and interpreting $x$ as the \emph{white}, distinguished point. We have introduced
$C_{n,1}(\overline{\D})$ because in $\T$ it may happen that $x\in\Sone$, whereas the numbers $z_i$ are always
in the interior of the unit disc; nevertheless the inclusion
$C_{n,1}(\D)\subset C_{n,1}(\overline{\D})$ is a homotopy equivalence.

The 2-fold covering $\Sq\colon \T\rightarrow C_{n,1}(\overline\D)$ has a nontrivial deck transformation
$\varepsilon\colon \T\rightarrow \T$, which corresponds to changing the sign of $y$.

\begin{lem}
 \label{lemma:varepsilon=id}
 The map $\varepsilon$ is homotopic to the identity of $\T$.
\end{lem}

\begin{proof}
 First we define a homotopy $H_{\varepsilon}\colon C_{n,1}(\overline\D)\times[0,1]\rightarrow C_{n,1}(\overline\D)$.
 For $p\in C_{n,1}(\overline\D)$ and $t\in [0,1]$ we set $H_{\varepsilon}(p,t)= e^{2\pi it}\cdot p$: that is,
 at time $t$ we rotate the configuration $p$ by an angle $2\pi t$ counterclockwise.
 Thus $H_\varepsilon$ is a homotopy from the identity of $C_{n,1}(\overline\D)$ to, again, the identity
 of $C_{n,1}(\overline\D)$.
 
 We lift this homotopy to a homotopy $\tilde H_{\varepsilon}\colon \T\times [0,1]\rightarrow \T$,
 starting from the identity of $\T$ at time $t=0$. At time $t=1$ any point $p\in \T$ is mapped to
 a point $p'$ lying over the same point of $C_{n,1}(\overline\D)$, i.e., $\Sq(p)=\Sq(p')$.
 
 During the homotopy $\tilde H_{\varepsilon}$ the complex number $y$ associated to $p$ is multiplied
 by $e^{2\pi i nt/2}$ at time $t$, since its square is multiplied by $e^{2\pi int}$. So at
 time $t=1$, the value of $y$ has been multiplied by $e^{2\pi in/2}=-1$, i.e., $p'=\varepsilon(p)$:
 here we have used that $n$ is odd.
\end{proof}

We get the following corollary in homology:

\begin{cor}
 \label{cor:twoproperties}
The map $\Sq_*\colon H_*(\T)\rightarrow H_*\pa{C_{n,1}(\overline\D)}$ has the following properties:
\begin{itemize}
 \item every element in the kernel of $\Sq_*$ has order 2 in $H_*\pa{\T}$;
 \item every element of the form $2c$ with $c\in H_*\pa{C_{n,1}(\overline\D)}$ is in the image of $\Sq_*$.
\end{itemize}
The same two properties hold for the transfer homomorphism $\Sq^!\colon H_n(C_{n,1}(\overline\D))\rightarrow H_n(\T)$.
\end{cor}
\begin{proof}
 We know that $\varepsilon_*$ is the identity map on $H_*\pa{\T}$, by lemma \ref{lemma:varepsilon=id};
 then $\Sq_*\circ\Sq^!$ is multiplication by 2, and $\Sq^!\circ\Sq_*$ is the sum of the identity and $\varepsilon_*$,
 so it is also multiplication by 2. The result follows immediately.
\end{proof}
There is a copy of $C_n(\D)$ embedded in $C_{n,1}(\overline{\D})$, given by selecting $1\in\Sone$ as white point:
this is exactly the image under $\Sq$ of the copy of $C_n(\D)$ embedded in $\T$ along $\phi^*s_0$. We get a diagram
of split short exact sequences
\[
 \begin{tikzcd}[column sep=3em, row sep=3em]
H_k(C_n(\D))\ar[r,"s"] \ar[d, equal] & H_k(\T)\ar[r] \ar[d, "\Sq_*"] & H_k(\T,C_n(\D)) \ar[d, "\Sq_*"] \\
H_k(C_n(\D))\ar[r,"s"] & H_k(C_{n,1}(\overline\D))\ar[r] & H_k(C_{n,1}(\overline\D),C_n(\D)).
 \end{tikzcd}
\]
Splitting is due to the fact that both $\T$ and $C_{n,1}(\overline{D})$ retract onto $C_n(\D)$: the retraction
is given by forgetting all data but the position of the $z_i$'s, so the left square is endowed with retractions
which are also compatible with the vertical maps.
Therefore the properties listed in corollary \ref{cor:twoproperties} hold also for the
map $\Sq_*\colon H_*(\T,C_n)\to H_*(C_{n,1}(\overline\D),C_n(\D))$.

Let $\mu\colon C_{n-1,1}(\D)\times\Sone\to C_{n,1}(\D)$ be the following map:
\[
 \mu\pa{\pa{\set{z_1,\dots,z_{n-1}},x},\theta}=\pa{\set{z_1,\dots,z_{n-1}.x+\delta\theta},x},
\]
where
\[
\delta=\delta(\set{z_1,\dots,z_{n-1}},x)=\frac 12\min\pa{\set{1-\abs{x}}\cup \set{\abs{z_i-x}\,|\,1\leq i\leq n-1}}>0.
\]
In words, $\mu$ transforms a configuration of one white point $x$ and $n-1$ black points $z_1,\dots, z_{n-1}$ into a configuration
with one more black point, by adding a new black point near $x$, in the direction of $\theta$. If we see $\Sone$ as a
homotopy equivalent replacement of $C_{1,1}$, then $\mu$ is up to homotopy a special case of the multiplication
$\mu\colon C_{1,h}\times C_{1,k}\to C_{1,h+k}$ making $\coprod_{k\geq 0} C_{1,k}$ into a $H$-space; we will not need
this general construction, which was first described in \cite{Vershinin}.

We recall also the following result, that can be found in \cite{Vassiliev}
\begin{lem}
\label{lem:Vassiliev}
Let $\nu$ be the composition
\[
  H_{k-1}\pa{C_{n-1,1}(\D)}\otimes H_1(\Sone)\subset H_k(C_{n-1,1}(\D)\times\Sone)
  \overset{\mu_*}{\rightarrow} H_k(C_{n,1}(\D)) \simeq
   H_k(C_{n,1}(\overline{\D}))
\]
Then $\nu$ is an isomorphism of $H_{k-1}\pa{C_{n-1,1}(\D)}\otimes H_1(\Sone)$ with the kernel
of the retraction $H_k(C_{n,1}(\overline{\D}))\to H_k(C_n(\D))$;
this kernel is also isomorphic to the group $H_k(C_{n,1}(\overline{\D}),C_n(\D))$.
\end{lem}

The following lemma analyses the behaviour of the map $\iota$ appearing in the Mayer-Vietoris sequence of lemma \ref{lemma:MVsuTZ}.

\begin{lem}
 \label{lemma:behaviourofiota}
Let $\iota$ be the map in the Mayer Vietoris sequence of lemma \ref{lemma:MVsuTZ}. We consider the
restriction of $\iota$ to the two summands of its domain, and its projection to the two summands of its codomain:
\begin{itemize}
 \item $\iota$ induces an isomorphism $H_k(C_{n-1,1}(\D))\rightarrow H_k(C_{n-1,1}(\D))$;
 \item $\iota$ induces the zero map $H_{k-1}(C_{n-1,1}(\D))\otimes H_1(\Sone)\rightarrow H_k(C_{n-1,1}(\D))$;
 \item $\iota$ induces the following map $H_{k-1}(C_{n-1,1}(\D))\otimes H_1(\Sone)\rightarrow H_k(\T,C_n(\D))$
 \[
 H_{k-1}(C_{n-1,1}(\D))\otimes H_1(\Sone) \overset{\nu}{\rightarrow} H_k(C_{n,1}(\overline{\D})
 \overset{\Sq^!}{\rightarrow} H_k(\T) \rightarrow H_k(\T,C_n(\D)).
 \]

\end{itemize}
\end{lem}
\begin{proof}
The first two points of the statement come from the behaviour of the map
$\iota\colon H_k(C_{n-1,1}(\D)\times \Sone)\rightarrow H_k(C_{n-1,1}(\D)\times \overline{\D})$ on Kunneth summands.

For the third point, recall that $C_{n-1,1}(\D)\times \Sone$ represents $\partial N(\Z)$, where $N(\Z)\cong\Z\times\overline{\D}$
is a tubular neighborhood of $\Z\simeq C_{n-1,1}(\D)$ in $\V$. Note that the map
$\Sq\colon \T\rightarrow C_{n,1}(\overline{\D})$ extends to a map (which is no longer a covering)
$\Sq\colon \V\rightarrow C_n(\D)\times \overline{\D}$: this map still consists in forgetting $y$.
Let $\Z'\subset C_n\times\overline \D$ be the subspace where the white point (in $\D$)
coincides with one of the $n$ black points; then again $\Z'\simeq C_{n-1,1}(\D)$ and $\Z'$ has a small, closed
tubular neighborhood
$N(\Z')\simeq \Z'\times\overline{\D}\subset C_n(\D)\times\overline{\D}$.

We can choose $N(\Z)$ to be $\Sq^{-1}(N(\Z'))\subset\V$; the map $\Sq\colon N(\Z)\to N(\Z')$ is a 2-fold branched covering,
and it is branched exactly over $\Z'$, which is homeomorphically covered by $\Z$. The restriction
$\Sq\colon \partial N(\Z)\to\partial N(\Z')$ is a genuine 2-fold covering, and agrees with the projections of these
boundaries of tubular neighborhoods on $\Z$ and $\Z'$ respectively (the projection of $N(\Z)$ onto $\Z$ can be chosen
to be the lift of the projection of $N(\Z')$ onto $\Z'\cong\Z$). So we have a commutative diagram
\[
 \begin{tikzcd}[column sep=3em, row sep=3em]
  \partial N(\Z)\simeq \Z\times\Sone \ar[r,"\pi_N"] \ar[d,"\Sq"] & \Z\cong C_{n-1,1}(\D)\ar[d,equal]\\
  \partial N(\Z')\simeq \Z'\times \Sone \ar[r,"\pi_{N'}"] & \Z'\cong C_{n-1,1}(\D)
 \end{tikzcd}
\]
In particular the composition $\pi_{N'}\circ\Sq$ is equal, up to identifying both $\Z$ and $\Z'$ with $C_{n-1,1}(\D)$,
to the map $\pi_{N}$, and in homology we can express the Gysin map $\pi_{N}^!$ as $\Sq^!\circ\pi_{N'}^!$.

We now observe that $\pi_N^!\colon H_{k-1}(\Z)\to H_k(\partial N(\Z))$ is exactly the inclusion of the summand
$H_{k-1}(C_{n-1,1})\otimes H_1(\Sone)\subset H_k(\partial N(\Z))$. The map $\iota$ is the composition of this inclusion
with the maps $H_k(\partial N(\Z))\to H_k(\T)$ induced by $\partial N(\Z)\subset\T$, and then the natural map
$H_k(\T)\to H_k(\T,C_n(\D))$.
 
 On the other hand $\pi_{N'}^!\colon H_{k-1}(\Z')\to H_k(\partial N(\Z'))$ is the inclusion
 \[
 H_{k-1}(C_{n-1,1}(\D))\otimes H_1(\Sone)\subset H_k(\partial N(\Z'))\simeq H_k(C_{n-1,1}(\D)\times\Sone);
 \]
 and the map
 $\nu\colon H_{k-1}(C_{n-1,1}(\D))\otimes H_1(\Sone)\to H_k(C_{n,1}(\D))$ is exactly this inclusion,
followed by the map $\mu_*\colon H_k(\partial N(\Z'))\to H_k(C_{n,1}(\overline{\D})$.
\end{proof}

We are now ready to prove theorem \ref{thm:fourtorsion}. We pick any class $a\in H_k(\V,C_n(\D))$,
and map it to $H_{k-1}(\partial N(\Z))$ along the long exact sequence of lemma \ref{lemma:MVsuTZ};
we get some class $b+c$, where $b\in H_{k-1}(C_{n-1,1}(\D))$ and $c\in H_{k-2}(C_{n-1,1}(\D))\otimes H_1(\Sone)$.
Then $\iota(b+c)$ must be zero, hence its first component, lying in $H_{k-1}(C_{n-1,1}(\D))$, must be zero;
therefore $b=0$ by the first two points of lemma \ref{lemma:behaviourofiota}.

Similarly $\iota(c)=0$, so also $\Sq_*\circ\iota(c)=0\in H_{k-1}(C_{n,1}(\overline{\D},C_n(\D))$; by the third point
of lemma \ref{lemma:behaviourofiota} this is equal to the image of $c$ under the map
\[
\begin{array}{c}
 H_{k-2}(C_{n-1,1}(\D))\otimes H_1(\Sone) \overset{\nu}{\rightarrow} H_{k-1}(C_{n,1}(\overline{\D}) \rightarrow\\
 \rightarrow H_{k-1}(C_{n,1}(\overline{\D}),C_n(\D)) \overset{\cdot 2}{\rightarrow} H_{k-1}(C_{n,1}(\overline{\D}),C_n(\D))
 \end{array}
 \]
As the composition of the first two maps is an isomorphism (see lemma \ref{lem:Vassiliev}) and clearly
multiplication by $2$ commutes with any maps of abelian groups, we have that $2c=0$.

Therefore $2a$ is in the kernel of the map $H_k(\V,C_n(\D))\rightarrow H_{k-1}(\partial N(\Z))$, so
it is in the image of the map $H_k(C_{n-1,1}(\D))\oplus H_k(\T,C_n(\D))\rightarrow H_k(\V,C_n(\D))$.

Let $d+e\mapsto 2a$, where $d\in H_k(C_{n-1,1}(\D))$ and $e\in H_k(\T,C_n(\D))$: we want now to show
that $2d+2e$ is in the image of $\iota$. Since $\iota(d+0)=d+h$ for some $h\in H_k(\T,C_n(\D))$, we
have $\iota(2d+0)=2d+2h$ so it suffices to find $i\in H_{k-1}(C_{n-1,1}(\D))\otimes H_1(\Sone)$
such that $\iota(i)=2e-2h$.

As $2e-2h=2(e-h)$ is twice an element in $H_k(\T,C_n(\D))\subset H_k(\T)$, by corollary
\ref{cor:twoproperties} there is an element $j+i\in H_k(C_{n,1}(\overline{\D}))$ such that
$\Sq^!(j+i)=2e-2h$, for some $j\in H_k(C_n(\D))$ and some $i\in H_{k-1}(C_{n-1,1}(\D))\otimes H_1(\Sone)
\simeq H_k(C_{n,1}(\overline{\D}),C_n(\D))$
(using again lemma \ref{lem:Vassiliev}). We now observe that the composition
\[
 H_k(C_n(\D)) \rightarrow H_k(C_{n,1}(\overline{\D})) \overset{\Sq^!}{\rightarrow} H_k(\T)
 \]
is equal to the composition
\[
  H_k(C_n(\D)) \overset{\cdot 2}{\rightarrow}  H_k(C_n(\D)) \subset H_k(\T),
\]
and in particular its image lies in the summand $H_k(C_n(\D)) \subset H_k(\T)$.

Indeed the covering $\Sq$ is the trivial covering over $C_n(\D)\subset C_{n,1}(\overline{\D})$,
with sections $\phi^*s_0\colon C_n(\D)\rightarrow \T$ and $\varepsilon\circ \phi^*s_0\colon C_n(\D)\rightarrow \T$,
and these sections are homotopic as maps $C_n(\D)\rightarrow \T$ by lemma \ref{lemma:varepsilon=id}.

Therefore we must have $\Sq^{!}(j)=0$ and we may assume $j=0$. It follows that $\iota(i)=2e-2h$, so the class $2d+2e$ is
in the image of $\iota$ and must therefore also be in the kernel of the map
$H_k(C_{n-1,1}(\D))\oplus H_k(\T,C_n(\D))\rightarrow H_k(\V,C_n(\D))$: this exactly means that
$4a=0\in H_k(\V,C_n(\D))$, and theorem \ref{thm:fourtorsion} now follows from the isomorphism
$H_k(\V,C_n(\D))\simeq H_{k-1}(\Br_n;\phi^*\H)$.

\bibliographystyle{plain}
\bibliography{bibliography.bib}{}

\end{document}